\documentclass{article}

\usepackage{amsmath}
\usepackage{amsfonts}
\usepackage{amsthm}
\usepackage{amssymb}
\usepackage{mathrsfs}
\usepackage[all]{xy}
\usepackage{enumerate}

\theoremstyle{definition}
\newtheorem{definition}{Definition}[section]
\newtheorem{proposition}[definition]{Proposition}
\newtheorem{theorem}[definition]{Theorem}

\newtheorem{corollary}[definition]{Corollary}

\theoremstyle{remark}

\DeclareMathOperator{\Ext}{Ext}
\DeclareMathOperator{\Spec}{Spec}

\begin{document}

\title{Extensions of mixed Hodge modules and Picard-Fuchs equations
}

\author{Pedro L. del Angel R. \\CIMAT \and
        Jos\'e J. Hern\'andez C.  \\CONACYT CIMAT  \\
}

\date{}

\maketitle

\begin{abstract}
The normal function associated to algebraic cycles in higher Chow groups defines a differential equation. This Picard-Fuchs equation defines an extension of $D$-modules as well as an extension of local systems. In this paper, we show that both extensions define the same extension of mixed Hodge modules determined by the normal function. 
%\keywords{Algebraic cycles \and Normal functions \and Mixed Hodge modules}

\end{abstract}

\section{Introduction}\label{sec:intro}

When we study a family of higher algebraic cycles, i.e. elements in $\text{CH}^r(X,m)$, for $X$ a smooth projective variety, we can associate a differential equation to the family, given by the Picard-Fuchs operator on the normal function of a cycle. It turns out that the non vanishing of this equation is related to the problem of finding indecomposable cycles, see \cite{dAMS}. Thus, the properties of the differential equation can be studied in view of its relation with finding non trivial cycles. In \cite{delAngelMS}, they study its rationality and field of definition of its coefficients. In this paper we show how to relate the extension of $D$-modules that arises from the differential equation with the extension that comes from the normal function where the Picard-Fuchs operator is evaluated, using the Abel-Jacobi map and the theory of mixed Hodge modules.

Mixed Hodge modules provide a generalization of classical Hodge theory. They can be thought as perverse sheaves with a mixed Hodge structure. The category of mixed Hodge modules was defined by M. Saito in \cite{SaitoM3}, \cite{SaitoM5}. In this work we are mainly interested in the Abel Jacobi map associated to higher Chow cycles. It turns out that this map can be factored through an extension in the category of mixed Hodge modules, thus providing an interpretation of normal functions as sections of mixed Hodge modules.

When dealing with admissible smooth families and higher algebraic cycles defined over a whole family, one has the Abel Jacobi map on the total space but also on the fibres, so one can think on both, the extension of MHM associated to the total space as well as the corresponding extensions associated to the fibres. In the later case, the normal functions can be thought of as sections on some smooth family of intermediate jacobians and they are related to non-homogeneous Picard-Fuchs equations.

\section{Extensions of Mixed Hodge Modules}

Let $X$ be a smooth variety  over $\mathbb{C}$. Given any $D_X$-module $\mathcal{M}$, via the de Rham complex, we have the de Rham functor from $D_X$-modules to complexes. Moreover, we can say the following:

\begin{theorem}[Riemann-Hilbert correspondence]
Let $X$ be a smooth variety over $\mathbb{C}$. Then the de Rham functor $\mathcal{M}\mapsto DR(\mathcal{M})$ induces an equivalence of categories between the category of holonomic $D_X$-modules with regular singularities and the category of perverse sheaves Perv$(\mathbb{C}_X)$.
\end{theorem}

Let PervW$(\mathbb{C}_X)$ be the category of weight filtered perverse sheaves in Perv$(\mathbb{C}_X)$ and $\text{MFW}_{rh}(X)$ be the category of regular holonomic $D_X$-modules with a good filtration together with a weight filtration. There is a natural functor from $\text{MFW}_{rh}(X)$ to PervW$(\mathbb{C}_X)$. It takes the pair $(\mathcal{M},W_{\mathcal{M}})$ to the pair $(DR(\mathcal{M}),DR(W_{\mathcal{M}}))$. If $\text{MFW}_{rh}(X;\mathbb{Q})$ is the fibre product $\text{PervW}(\mathbb{Q}_X)\times_{\text{PervW}(\mathbb{C}_X)} \text{MFW}_{rh}(X)$, its objects are of the form $(K^{\bullet}_\mathbb{Q},W_{\mathbb{Q}},\mathcal{M},F,W,\alpha)$ where $K^{\bullet}_\mathbb{Q}$ is a perverse sheaf with weight filtration $W_{\mathbb{Q}}$, $\mathcal{M}$ is a holonomic $D_X$-module with regular singularities and weight filtration $W$, $F$ is a good filtration on $\mathcal{M}$ and $\alpha$ is an isomorphism of filtered objects, i.e. $\alpha(K^{\bullet}_\mathbb{Q}\otimes\mathbb{C},W_\mathbb{Q}\otimes\mathbb{C})\simeq(DR(\mathcal{M}),DR(W_{\mathcal{M}}))$, where $DR$ is the de Rham functor.

\begin{theorem}[M. Saito]
For any smooth variety $X$ over $\mathbb{C}$ there exists an abelian category $\text{MHM}(X)$ that is a full subcategory of $\text{MFW}_{rh}(X;\mathbb{Q})$. $\text{MHM}(X)$ is called the category of mixed Hodge modules.
\end{theorem}

The category of mixed Hodge modules contains a semi-simple full subcategory of modules of pure weight. These are called polarizable Hodge modules. In the bounded derived category $D^b(\text{MHM}(X))$ all expected operations are defined: $f_{\ast},f^{\ast}, f_{!},f^{!},\mathbb{D}$, etc. $\text{MHM}(\Spec ({\mathbb{C}}))$ is isomorphic to the category of graded polarizable mixed Hodge structures. For $k$ a subfield of $\mathbb{C}$, the category $\text{MHM}(X)$ of mixed Hodge modules of a smooth variety $X$ over $k$ is well defined.

Let $r\in \mathbb{Z}$. The $D_X$-module $\mathcal{O}_X$ has a good filtration defined by $Gr^F_{r}(\mathcal{O}_X)=\mathcal{O}_X$. Since $DR(\mathcal{O}_X)$ is quasi-isomorphic to $\mathbb{C}_X[\text{dim}X]$, we get a mixed Hodge module $\mathbb{Q}_X(r)[\text{dim}X]$ given by 
\[
(\mathbb{Q}_X(r)[\text{dim}X],W_{\mathbb{Q}},\mathcal{O}_X,F, W, \alpha),
\]
 where $Gr_{\text{dim}X-2r}^{W_{\mathbb{Q}}}(\mathbb{Q}_X(r)[\text{dim}X])=\mathbb{Q}_X(r)[\text{dim}X]$ and $W$ and $\alpha$ have the only possible interpretation. This is called the constant (or Tate) mixed Hodge module. We usually use the shifted module \[\mathbb{Q}_X(r)=\mathbb{Q}_X(r)[\text{dim}X][-\text{dim}X].\]

The following proposition and its corollary are essential to construct the maps we are looking for.
\begin{proposition}[\cite{SaitoM4}]\label{dec}
Let $f:X\to Y$ be a proper morphism of smooth varieties. If $M\in \text{MHM}(X)$ is pure, we have a non canonical isomorphism
\[
f_{\ast}M\simeq\bigoplus_{i}H^{i}f_{\ast}M[-i]
\]
in $D^b\text{MHM}(Y)$.
\end{proposition}

\begin{corollary}\label{ssmhm}
Let $f:X\to Y$ be a proper smooth morphism of quasi-projective smooth varieties over $k\subset \mathbb{C}$. Then there is a Leray spectral sequence
\[
E^{p,q}_{2}=\Ext^{p}_{\text{MHM}(Y)}(\mathbb{Q}_{Y}(0),R^{q}f_{\ast}M)\Rightarrow \Ext_{\text{MHM}(X)}^{p+q}(\mathbb{Q}_{X}(0),M),
\]
which degenerates at $E_2$, for any pure $M\in \text{MHM}(X)$.
\end{corollary}
\begin{proof}
The existence of the Leray filtration on $\Ext_{\text{MHM}(X)}^{p+q}(\mathbb{Q}_{X}(0),M)$ such that 
\[
E^{p,q}_{2}=\Ext^{p}_{\text{MHM}(Y)}(\mathbb{Q}_{Y}(0),R^{q}f_{\ast}M)
\]
is clear. Then we use the decomposition in proposition \eqref{dec} and apply Deligne's criterion to conclude that the Leray spectral sequence associated to the Leray filtration degenerates at $E_2$.
\end{proof}

This spectral sequence induces a canonical Leray filtration on $\Ext_{\text{MHM}(X)}^{p+q}(\mathbb{Q}_{X}(0),M)$. 

We also have the following short exact sequence:

\begin{proposition}\label{sequence}
Let $Y$ be a smooth quasiprojective variety over $k\subset \mathbb{C}$, $M\in \text{MHM}(Y)$ and $g:Y\to \Spec(k)$ be the natural morphism. Then there exists a short exact sequence
\begin{multline*}
0\to \Ext^{1}_{\text{MHM}(\Spec(k))}(\mathbb{Q}_{\Spec(k)}(0), R^{q-1}g_{\ast}M)
\to \\\Ext^{q}_{\text{MHM}(Y)}(\mathbb{Q}_Y(0), M)\to
\hom_{\text{MHM}(\Spec(k))}(\mathbb{Q}_{\Spec(k)}(0), R^{q}g_{\ast}M)\to 0
\end{multline*}
\end{proposition}
\begin{proof}
By Corollary \eqref{ssmhm} there is the Leray spectral sequence
\[
E^{p,q}_{2}=\Ext^{p}_{\text{MHM}(\Spec(k))}(\mathbb{Q}_{\Spec(k)}(0),R^{q}g_{\ast}M)\Rightarrow \Ext_{\text{MHM}(Y)}^{p+q}(\mathbb{Q}_{Y}(0),M),
\]
which degenerates at $E_{2}$. Then, if $L$ is the Leray filtration on $\Ext^{p+q}_{\text{MHM}(Y)}$,
\[
E^{p,q}_{2}=Gr^p_L\Ext_{\text{MHM}(Y)}^{p+q}(\mathbb{Q}_{Y}(0),M).
\]
On the other hand, $\text{MHM}(\Spec(k))$ is a subcategory of the category of mixed Hodge structures (see next section), and we know that $\Ext^{\ell}_{\text{MHS}}=0$ for all $\ell\geq 2$, and similarly for $\text{MHM}(\Spec(k))$. Then
\begin{equation*}
\begin{split}
L^1\Ext_{\text{MHM}(Y)}^{q}(\mathbb{Q}_{Y}(0),M)& =Gr^1_L\Ext_{\text{MHM}(Y)}^{q}(\mathbb{Q}_{Y}(0),M)\\
& =\Ext^{1}_{\text{MHM}(\Spec(k))}(\mathbb{Q}_{\Spec(k)}(0),R^{q-1}g_{\ast}M).
\end{split}
\end{equation*}
Since  
\[
\hom_{\text{MHM}(\Spec(k))}(\mathbb{Q}_{\Spec(k)}(0), R^{q}g_{\ast}M)=Gr^0_L\Ext_{\text{MHM}(Y)}^{q}(\mathbb{Q}_{Y}(0),M),
\]
we get that $\hom_{\text{MHM}(\Spec(k))}(\mathbb{Q}_{\Spec(k)}(0), R^{q}g_{\ast}M)$ is the quotient of \\ $\Ext_{\text{MHM}(Y)}^{q}(\mathbb{Q}_{Y}(0),M)$ and $\Ext^{1}_{\text{MHM}(\Spec(k))}(\mathbb{Q}_{\Spec(k)}(0),R^{q-1}g_{\ast}M)$, and this implies that we have the short exact sequence required.
\end{proof}

\section{Maps for Higher Chow Groups}

Let $X$ be a smooth variety over $k\subset \mathbb{C}$. Let´s denote the $r$-th higher Chow group of $X$ tensored with $\mathbb{Q}$ by $\text{CH}^{r}(X,m;\mathbb{Q})$. There is a cycle class map from this group to an extension of mixed Hodge modules. It was constructed by  M. Saito (see \cite{As} as well).

\begin{theorem}[M. Saito]\label{ccm}
Let $X$ be a smooth variety over $k\subset \mathbb{C}$. Then there exists a cycle map 
\[
c_{r,m}:\text{CH}^{r}(X,m;\mathbb{Q})\to \Ext_{\text{MHM}(X_{\mathbb{C}})}^{2r-m}(\mathbb{Q}_{X_{\mathbb{C}}}(0),\mathbb{Q}_{X_{\mathbb{C}}}(r)). 
\]
\end{theorem}

If we consider the  natural morphism $g:X_{\mathbb{C}}\to \Spec(\mathbb{C})$, we can use lemma \eqref{sequence} to get a short exact sequence
\begin{multline*}\label{ses1}
0\to \Ext^{1}_{\text{MHM}(\Spec(\mathbb{C}))}(\mathbb{Q}_{\Spec(\mathbb{C})}(0), R^{2r-m-1}g_{\ast}\mathbb{Q}_{X_{\mathbb{C}}}(r))\\
\to \Ext^{2r-m}_{\text{MHM}(X_{\mathbb{C}})}(\mathbb{Q}_{X_{\mathbb{C}}}(0), \mathbb{Q}_{X_{\mathbb{C}}}(r))\to \\
\hom_{\text{MHM}(\Spec(\mathbb{C}))}(\mathbb{Q}_{\Spec(\mathbb{C})}(0), R^{2r-m}g_{\ast}\mathbb{Q}_{X_{\mathbb{C}}}(r))\to 0
\end{multline*}
for $X$ a smooth quasiprojective variety over $k\subset \mathbb{C}$. 
Since $\text{MHM}(\Spec ({\mathbb{C}}))$ is isomorphic to the category of graded polarizable mixed Hodge structures, we have the isomorphisms
\[
\mathbb{Q}_{\Spec(\mathbb{C})}(0)\simeq \mathbb{Q}(0),
\]
\[
R^{2r-m-1}g_{\ast}\mathbb{Q}_{X_{\mathbb{C}}}(r)\simeq H^{q-1}(X,\mathbb{Q}(r)),
\]
\[
R^{2r-m}g_{\ast}\mathbb{Q}_{X_{\mathbb{C}}}(r)\simeq H^{q}(X,\mathbb{Q}(r)). 
\]
where $H^{\bullet}(X,\mathbb{Q}(r))$  means $H^{\bullet}(X_{\mathbb{C}}^{\text{an}},\mathbb{Q}(r))$, a notation we will continue to use in this paper. Then we can rewrite the short exact sequence as
\begin{multline}\label{extandhom}
0\to \Ext^{1}_{\text{MHS}}(\mathbb{Q}(0), H^{2r-m-1}(X,\mathbb{Q}(r)))
\to \\\Ext^{2r-m}_{\text{MHM}(X_{\mathbb{C}})}(\mathbb{Q}_{X_{\mathbb{C}}}(0), \mathbb{Q}_{X_{\mathbb{C}}}(r))\to
\hom_{\text{MHS}}(\mathbb{Q}(0), H^{2r-m}(X,\mathbb{Q}(r)))\to 0
\end{multline}
Using the  morphism of Theorem \eqref{ccm},
\[
c_{r,m}:\text{CH}^{r}(X,m;\mathbb{Q})\to \Ext_{\text{MHM}(X_{\mathbb{C}})}^{2r-m}(\mathbb{Q}_{X_{\mathbb{C}}}(0),\mathbb{Q}_{X_{\mathbb{C}}}(r)), 
\]
we have a map
\begin{equation}\label{map1}
\text{CH}^{r}(X,m;\mathbb{Q})\to \hom_{\text{MHS}}(\mathbb{Q}(0), H^{2r-m}(X,\mathbb{Q}(r))).
\end{equation}
We denote by $[\xi]$ the image of $\xi \in  \text{CH}^{r}(X,m;\mathbb{Q})$ in $\hom_{\text{MHS}}(\mathbb{Q}(0), H^{2r-m}(X,\mathbb{Q}(r)))$
and if this image is zero (this is always the case when $X$ is complex projective and $m>0$) then $\xi$ maps to an element of $\Ext^{1}_{\text{MHS}}(\mathbb{Q}(0),H^{2r-m-1}(X,\mathbb{Q}(r)))$ as we can see from \eqref{extandhom}. We call this map $AJ$. The kernel of the \eqref{map1} is denoted by $\text{CH}^{r}_{\text{hom}}(X,m;\mathbb{Q})$. 

\begin{definition}\label{aj}
The map
\[
AJ:\text{CH}^{r}_{\text{hom}}(X,m;\mathbb{Q})\to \Ext^{1}_{\text{MHS}}(\mathbb{Q}(0),H^{2r-m-1}(X,\mathbb{Q}(r))).
\]
is called the Abel-Jacobi map.
\end{definition}

This definition of $AJ(\xi)$ coincides with the definition of higher Abel-Jacobi maps defined by integration of currents when $X$ is smooth projective over $\mathbb{C}$; it is proved by Kerr, Lewis and Muller-Stach in \cite{KLMS}, see also \cite{KL}.

\section{The extension associated to a family}

Let $X$ be a smooth projective variety over $\mathbb{C}$, $g:X\to\mathbb{C}$ the natural morphism. Then, in the previous section we constructed a map
\begin{equation}\label{ajtomhm}
\xymatrix{
\text{CH}^{r}_{\text{hom}}(X,m;\mathbb{Q}) \ar[dr]_{AJ} \ar[r]& \Ext^{1}_{\text{MHM}(\Spec(\mathbb{C}))}(\mathbb{Q}_{\Spec(\mathbb{C})}(0), R^{2r-m-1}g_{\ast}\mathbb{Q}_X(r))\ar[d]^{=}\\
& \Ext^{1}_{\text{MHS}}(\mathbb{Q}(0),H^{2r-m-1}(X,\mathbb{Q}(r))).
}
\end{equation}
Thus, to any $\xi \in \text{CH}^{r}_{\text{hom}}(X,m;\mathbb{Q})$ we can associate an extension
\[
\mathbf{E}\in \Ext^{1}_{\text{MHS}}(\mathbb{Q}(0),H^{2r-m-1}(X,\mathbb{Q}(r))).
\]
Since we are working in the category of mixed Hodge modules, we should note that these extension groups are defined as extensions in the categorical sense. For our interests, we would like to have a different point of view. More precisely, that $\mathbf{E}$ is an extension in the sense of Yoneda is a consequence of the following theorem of Verdier, \cite{Verdier} (see also Gelfand-Manin \cite{GelfandManin} or Peters-Steenbrink \cite{PetersSteenbrink} ):
\begin{proposition}\label{equivalenceext}
If  $\mathcal{A}$ is an abelian category then for any $A$ and $B$ objects in $\mathcal{A}$, $\Ext^n(A,B)$  is the same as the Yoneda extension group.
\end{proposition} 
Therefore $\mathbf{E}$ is an extension (more precisely an equivalence class of extensions) of mixed Hodge structures of the form:
\[
0\to H^{2r-m-1}(X,\mathbb{Q}(r))\to \mathbf{E}\to \mathbb{Q}(0)\to 0.
\]
As we can see in \eqref{ajtomhm}, $\mathbf{E}$ is automatically a mixed hodge module over $\Spec(\mathbb{C})$.

Now, let's work in the relative case, i.e. let's consider a family of smooth proyective varieties over $\mathbb{C}$, 
\begin{eqnarray} \label{familia}
\xymatrix{
X_t \ar[d] \ar[r] & \mathcal{X} \ar[d]^{f}\\
\Spec(\mathbb{C}) \ar[r]^(0.6){t} & S
}
\end{eqnarray}
where $f$ is smooth and proper with fibre $X_t=f^{-1}(t)$, for  $t\in S$ and $\mathcal{X}$, $S$ are defined over $k\subset \mathbb{C}$ algebraically closed. In the family each fibre is of the form $X_t=\mathcal{X}\times_S\Spec (\mathbb{C})$, where $k(t)=\mathbb{C}$ and we have natural morphisms $\Spec(\mathbb{C})\to S$, $X_t\to \Spec(\mathbb{C})$. 

From Corollary \eqref{ssmhm}, for $f:X\to Y$  a proper smooth morphism of quasi-projective smooth varieties over $k\subset \mathbb{C}$ there is a Leray spectral sequence
\[
E^{p,q}_{2}=\Ext^{p}_{\text{MHM}(Y)}(\mathbb{Q}_{Y}(0),R^{q}f_{\ast}M)\Rightarrow \Ext_{\text{MHM}(X)}^{p+q}(\mathbb{Q}_{X}(0),M),
\]
which degenerates at $E_2$, for any pure $M\in \text{MHM}(X)$. In particular we have a Leray filtration $L$ on $\Ext$. For $f:\mathcal{X}\to S$ a family of varieties over $\mathbb{C}$, then we have a map
\begin{multline*}
L^1\Ext^{2r-m}_{\text{MHM}(\mathcal{X})}(\mathbb{Q}_{\mathcal{X}}(0), \mathbb{Q}_{\mathcal{X}}(r))\to Gr^1_L\Ext^{2r-m}_{\text{MHM}(\mathcal{X})}(\mathbb{Q}_{\mathcal{X}}(0), \mathbb{Q}_{\mathcal{X}}(r)) \\
= \Ext^{1}_{\text{MHM}(S)}(\mathbb{Q}_{S}(0), R^{2r-m-1}f_{\ast}\mathbb{Q}_{\mathcal{X}}(r)).
\end{multline*}
where $Gr^1_L=L^1/L^2$. 
$L$ also induces a filtration $F_L$ on $\text{CH}^{r}(\mathcal{X},m;\mathbb{Q})$ given by \[F_{L}^{j}\text{CH}^{r}(\mathcal{X},m;\mathbb{Q}):= c_{r,m}^{-1}(L^{j}\Ext^{2r-m}_{\text{MHM}(X)}(\mathbb{Q}_{\mathcal{X}}(0), \mathbb{Q}_{\mathcal{X}}(r))),\] where $c_{r,m}$ is the cycle class map from Theorem \eqref{ccm}. Using the map above and the cycle class map we get
\[
F_L^1\text{CH}^{r}(\mathcal{X},m;\mathbb{Q})\to \Ext^{1}_{\text{MHM}(S)}(\mathbb{Q}_{S}(0), R^{2r-m-1}f_{\ast}\mathbb{Q}_{\mathcal{X}}(r)).
\]

Since the Leray spectral sequence is functorial, then from the diagram \ref{familia} and the pullback on $X_t\to {\mathcal{X}}$ we get the diagram
\begin{small}
\begin{equation}\label{diag1}
\xymatrix{
F_L^1\text{CH}^{r}(\mathcal{X},m;\mathbb{Q}) \ar[d] \ar[r] & F_L^1\text{CH}^{r}(X_t,m;\mathbb{Q}) \ar[d]\\
\Ext^{1}_{\text{MHM}(S)}(\mathbb{Q}_{S}(0), R^{2r-m-1}f_{\ast}\mathbb{Q}_{\mathcal{X}}(r)) \ar[r] &  \Ext^{1}_{\text{MHM}(\Spec(\mathbb{C}))}(\mathbb{Q}_{\Spec(\mathbb{C})}(0), R^{2r-m-1}f_{\ast}\mathbb{Q}_{X_t}(r))
}
\end{equation}
\end{small}
We can relate this construction to the previous maps in the following:
\begin{proposition}
The right vertical map in \eqref{diag1} is the Abel-Jacobi map for $m\geq 1$.
\end{proposition}
\begin{proof}
We have an injective map
\[
Gr^0_{F_L}\text{CH}^{r}(X_t,m;\mathbb{Q})\hookrightarrow Gr^0_{L} \Ext^{2r-m}_{\text{MHM}(X_t)}(\mathbb{Q}_{X_t}(0), \mathbb{Q}_{X_t}(r)).
\]
But
\[
 Gr^0_{L} \Ext^{2r-m}_{\text{MHM}(X_t)}(\mathbb{Q}_{X_t}(0), \mathbb{Q}_{X_t}(r))=  \Ext^{0}_{\text{MHM}(\Spec(\mathbb{C}))}(\mathbb{Q}_{\Spec(\mathbb{C})}(0), R^{2r-m}f_{\ast}\mathbb{Q}_{X_t}(r))
\]
and
\[
\Ext^{0}_{\text{MHM}(\Spec(\mathbb{C}))}(\mathbb{Q}_{\Spec(\mathbb{C})}(0), R^{2r-m}f_{\ast}\mathbb{Q}_{X_t}(r))=\hom_{\text{MHS}}(\mathbb{Q}(0), H^{2r-m}(X_t, \mathbb{Q}(r)))
\]

Recall from the morphism $\Spec(\mathbb{C}) \to S$ that we have a morphism \\MHM$(S)\to$MHM$(\Spec(\mathbb{C}))$
and MHM$(\Spec(\mathbb{C}))$ is the category MHS. That 
\[
\Ext^{1}_{\text{MHM}(\Spec(\mathbb{C}))}(\mathbb{Q}_{\Spec(\mathbb{C})}(0), R^{2r-m-1}f_{\ast}\mathbb{Q}_{X_t}(r))= \Ext^{1}_{\text{MHS}}(\mathbb{Q}(0), H^{2r-m-1}(X_t, \mathbb{Q}(r)))
\]
follows from this. Since $X_t$ is projective and because we are working in the category of mixed Hodge structures, for $m\geq 1$ by weight reasons:
\[
\hom_{\text{MHS}}(\mathbb{Q}(0), H^{2r-m}(X_t, \mathbb{Q}(r)))=0.
\]
Therefore
\[
Gr^0_{L} \Ext^{2r-m}_{\text{MHM}(X_t)}(\mathbb{Q}_{X_t}(0), \mathbb{Q}_{X_t}(r))=0
\]
so $L^0=L^1$ and $F_L^0=F_L^1$. This shows that $F_L^1\text{CH}^{r}(X_t,m;\mathbb{Q})=\text{CH}^{r}(X_t,m;\mathbb{Q})$. Also, from \eqref{extandhom}, for $X_t$ projective the map \eqref{map1} is always zero and therefore $\text{CH}^{r}(X_t,m;\mathbb{Q})=\text{CH}^{r}_{\text{hom}}(X_t,m;\mathbb{Q})$.
\end{proof}

In the proof above $F_L^0=F_L^1$ and $\text{CH}^{r}(X_t,m;\mathbb{Q})=\text{CH}^{r}_{\text{hom}}(X_t,m;\mathbb{Q})$ for any smooth projective variety when $m\geq 1$. So, we have a diagram 
\[
\xymatrix{
\text{CH}^{r}(\mathcal{X},m;\mathbb{Q}) \ar[d] \ar[r] & \text{CH}^{r}_{\text{hom}}(X_t,m;\mathbb{Q}) \ar[d]^{AJ}\\
\Ext^{1}_{\text{MHM}(S)}(\mathbb{Q}_{S}(0), R^{2r-m-1}f_{\ast}\mathbb{Q}_{\mathcal{X}}(r)) \ar[r] & \Ext^{1}_{\text{MHS}}(\mathbb{Q}(0), H^{2r-m-1}(X_t, \mathbb{Q}(r))).
}
\]

If we choose a cycle $\xi \in \text{CH}^{r}(\mathcal{X},m;\mathbb{Q})$ and $\xi_t=\xi|_{X_t}$ we get maps
\[
\text{CH}^{r}_{\text{hom}}(X_t,m;\mathbb{Q})\to \Ext^{1}_{\text{MHM}(\Spec(\mathbb{C}))}(\mathbb{Q}_{\Spec(\mathbb{C})}(0), R^{2r-m-1}f_{\ast}\mathbb{Q}_{\mathcal{X}_t}(r))
\]
and every $\xi_t\in \text{CH}^{r}_{\text{hom}}(X_t,m;\mathbb{Q})$ maps to an extension $\mathbf{E}_t$ of mixed Hodge modules.
Moreover, the morphism $\Spec{\mathbb{C}}\to S$ induces the bottom map in the diagram
\begin{footnotesize}
\[
\Ext^{1}_{\text{MHM}(S)}(\mathbb{Q}_{S}(0),R^{2r-m-1}f_{\ast}\mathbb{Q}_{\mathcal{X}}(r))\to \Ext^{1}_{\text{MHM}(\Spec(\mathbb{C}))}(\mathbb{Q}_{\Spec(\mathbb{C})}(0), R^{2r-m-1}f_{\ast}\mathbb{Q}_{\mathcal{X}_t}(r)).
\]
\end{footnotesize}

In summary, we have the following:
\begin{theorem}\label{teorema1}
To any smooth projective family $f:\mathcal{X}\to S$ over $\mathbb{C}$ and  cycle $\xi \in \text{CH}^{r}(\mathcal{X},m;\mathbb{Q})$, $m\geq 1$, we can associate a mixed Hodge module $\mathcal{E}$ that is an extension
\[
0\to R^{2r-m-1}f_{\ast}\mathbb{Q}_{\mathcal{X}}(r)\to \mathcal{E} \to \mathbb{Q}_{S}(0) \to 0
\]
of mixed Hodge modules, that restricts to the extensions given by the Abel-Jacobi map in the fibres.
\end{theorem}

The extension of mixed Hodge modules that captures the family is, precisely, in
\[
\Ext^{1}_{\text{MHM}(S)}(\mathbb{Q}_{S}(0), R^{2r-m-1}f_{\ast}\mathbb{Q}_{\mathcal{X}}(r)).
\]
Let $\text{MF}_{rh}(Y)$ denote the category of regular holonomic $D_Y$-modules with a good filtration, where $Y$ is a smooth variety over a closed field $k$ of characteristic $0$. We can map any mixed Hodge module to its corresponding filtered regular holonomic $D_Y$-module. In particular, for $f:\mathcal{X}\to S$, $\mathbb{Q}_{S}(0)$ maps to $\mathcal{O}_S$ and $R^{2r-m-1}f_{\ast}\mathbb{Q}_{\mathcal{X}}(r)$ to 
\[
\mathbb{R}^{2r-m-1}f_{\ast}\Omega ^{\bullet }_{\mathcal{X}/S}(r)=\mathcal{H}^{2r-m-1}_{DR}(\mathcal{X}/S)(r).\]
$\Omega ^{\bullet }_{\mathcal{X}/S}$ denotes the complex of relative differentials of $\mathcal{X}$ over $S$ and $\mathcal{H}^{q}_{DR}(\mathcal{X}/S)$ is the $q$-th relative de Rham cohomology sheaf of $\mathcal{X}$ over $S$ that by definition is
\[
\mathcal{H}^{q}_{DR}(\mathcal{X}/S):=\mathbb{R}^{q}f_{\ast}\Omega^{\bullet}_{\mathcal{X}/S}
\]
for any $\mathcal{X}$ and $S$ smooth quasiprojective varieties over $k$ and $f:\mathcal{X}\to S$ a proper smooth morphism. Then we have a natural map
\begin{small}\[
\Ext^{1}_{\text{MHM}(S)}(\mathbb{Q}_{S}(0),R^{2r-m-1}f_{\ast}\mathbb{Q}_{\mathcal{X}}(r)) \to \Ext^{1}_{\text{MF}_{\text{rh}}(S)}(\mathcal{O}_S,\mathcal{H}^{2r-m-1}_{DR}(\mathcal{X}/S)(r)).
\]\end{small}
The extension $\mathcal{E}$ (we denote the correponding $D_S$-module also by $\mathcal{E}$), can be seen as an extension of $D_S$-modules
\begin{eqnarray} \label{nose}
0\to \mathcal{H}^{2r-m-1}_{DR}(\mathcal{X}/S)(r) \to \mathcal{E} \to \mathcal{O}_S \to 0.
\end{eqnarray}

\section{Extensions and normal functions}

Let $\overline{\pi}:\overline{\mathcal{X}}\to \overline{S}$  be a flat morphism of relative dimension $k$ between proper, irreducible smooth varieties defined over $\mathbb{C}$. Assume dim $(\overline{S}) = 1$ and let $\mathcal{X}\subset \overline{\mathcal{X}}$, $S\subset \overline{S}$ be open dense subsets such that $\pi:=\overline{\pi}|_{\mathcal{X}}:\mathcal{X}\to S$ is smooth. If $\mathbb{H}$ is the local system over $S$ associated to the primitive part of $R^k\pi_*\mathbb{Q}$ and $\mathcal{H}=\mathbb{H}\otimes\mathcal{O}_S$, then the Gauss-Manin connection $\nabla:\mathcal{H}\to\mathcal{H}\otimes \Omega^1_S$ induces a differential equation
\begin{eqnarray}\label{PF}
D_{PF}\; f =0
\end{eqnarray}
called the Picard-Fuchs equation of the periods of the family $\mathcal{X}\to S$.

Let $\{\omega_1,\dots ,\omega_n\}$ be a local base of solutions of $\ref{PF}$, $g$ be an holomorphic function on $S$, $h$ be a solution of the differential equation
\begin{eqnarray}\label{nlPF}
D_{PF}\; h = g
\end{eqnarray}
in a neighborhood $U$ of a point $t_0\in S$ and $\gamma$ be a closed path on $S$ with starting point $t_0$. If we extend $h$ along $\gamma$, we will get back another solution of \ref{nlPF}, which then looks like
\[
\gamma(h)= h + \sum_i a_i\omega_i
\]
Let $h_1$ be another solution of \ref{nlPF} in $U,$ then $h_1-h$ will be a solution of \ref{PF} and therefore they differ by a linear combination ${\displaystyle \sum_i c_i\omega_i}$. If the action of the monodromy $\pi_1(S, t_0)$ on the periods is given by $\rho$, then $\gamma$ will act on the periods by $M=\rho(\gamma)$ and we can extend $h_1$ along $\gamma$ as well, getting
\[
\gamma(h_1) = h+\sum_i a_i\omega_i + \sum_i c_i \rho(\gamma)(\omega_i) = h_1 + \sum_i a_i\omega_i + \sum_i c_i (M-I)\omega_i
\]
so that we can associate to \ref{nlPF} a well defined class (see \cite{stiller}, part II section 3)
\[
\alpha=[\gamma\to a_{\gamma}]\in H^1(\pi_1(S,t_0), H_0) = H^1(S,\mathbb{H}),
\]
where $H_0$ is the fiber of $\mathbb{H}$ at $t_0$ and $a_{\gamma}=(a_1, \dots, a_n)$.

If $\xi \in \text{CH}^{r}(\mathcal{X},m;\mathbb{Q})$ and $\xi_t=\xi|_{X_t}$, the Abel-Jacobi map \eqref{aj} let us define:

\begin{definition}
The normal function $\nu_{\xi}:S\to {\mathcal J}$ associated to $\xi$ is given by
\[
\nu_{\xi}(t)=AJ(\xi_t),
\]
where $\mathcal{J}$ fits in the short exact sequence
\[
0\to R^{2r-m-1}\pi_*\mathbb{Z}\to R^{2r-m-1}\pi_*\mathbb{C}\otimes\mathcal{O}_S/ F^{r}(R^{2r-m-1}\pi_*\mathbb{C}\otimes\mathcal{O}_S)\to \mathcal{J}\to 0
\]
and
${\mathcal J}_t = \Ext^{1}_{\text{MHS}}(\mathbb{Q}(0), H^{2r-m-1}(X_t, \mathbb{Q}(r)))$.
\end{definition}

If $k=2r-m-1$, let $\overline{\nu_{\xi}}$ be a lifting of $\nu_{\xi}$ to $R^{k}\pi_*\mathbb{C}\otimes\mathcal{O}_S/ F^r(R^{k}\pi_*\mathbb{C}\otimes\mathcal{O}_S)$ and let us assume that $\overline{\nu_{\xi}}$ is not a solution of the differential equation \ref{PF} and the family $\overline{\mathcal{X}}\to \overline{S}$ is admisible (i.e. the monodromy representation at the fibers at infinity is irreducible and unipotent), then $g(t):=D_{PF}\overline{\nu_{\xi}}(t)$ defines a non zero holomorphic function which does not depend on the choice of the lifting (see \cite{delAngelMS}, thm. 1.1 and 3.2)  and we can associate to it (and therefore to the normal function $\nu_{\xi}$) a non zero class $\alpha_{\xi}$ in $H^1(S,\mathbb{H})$. Moreover, in this case the solutions of the equation $D_{PF}\nu_{\xi}(t)=g(t)$ are also solutions of the homogeneous equation
\begin{eqnarray}\label{ext-PF}
(\frac{d}{dt}-d\log g)D_{PF} h(t) = 0
\end{eqnarray}
which is the homogeneous equation associated to a local system $\mathbb{E}$. Observe that the solutions of the equation $D_{PF}h(t)=0$ are solutions of \ref{ext-PF} as well, therefore $\mathbb{E}$ is an extension of 
$\mathbb{H}$ of the form
\begin{eqnarray} \label{locsys}
0\to \mathbb{H}\to\mathbb{E}\to\mathbb{Q}(-r)\to 0
\end{eqnarray}
and we claim that this extension is determined by the class $\alpha_{\xi}$. 

Indeed, with the notation above, if the action of $\gamma$ on $\{w_1,\dots , w_n\}$ (a flat basis of $\mathcal{H}=\mathbb{H}\otimes\mathcal{O}_S$) is given by $M=\rho(\gamma)$, then the action of $\gamma$ on $\{w_1,\dots , w_n, \overline{\nu_{\xi}} \}$ (a flat basis of $\mathcal{E}=\mathbb{E}\otimes\mathcal{O}_S$) is given by 
\[
\left(\begin{matrix}
M & a_{\gamma}\cr
0 & 1
\end{matrix}\right)
\]
i.e., it is encoded in $\alpha_{\xi}$ so that we can recover the extension \ref{locsys} from it.

After tensoring \ref{locsys} with $\mathcal{O}_S$ we get the extension of $D_S$ modules (see \ref{nose})
\[
0\to \mathcal{H}^{2r-m-1}_{DR}(\mathcal{X}/S) \to \mathcal{E} \to \mathcal{O}_S(-r) \to 0
\]
associated to the local systems $\mathbb{H}, \mathbb{E}$ and $g\cdot \mathbb{Q}(-r)$. Thanks to theorem \eqref{teorema1}, we can think of the normal function $\nu_{\xi}$ as an extension of mixed Hodge modules whose underlying extension of $D_S$ modules is precisely the extension \ref{nose}. This reflects the fact that to a $\mbox{MHM}(S)$ one can associate both, a $D_S$-module and a perverse sheaf and so to an extension of $\mbox{MHM}(S)$ should correspond both, an extension of $D_S$-modules as well as an extension or perverse sheafs on $S$ and both are determined by the image of $\nu_{\xi}$ on the corresponding extension groups.
\begin{proposition}
The extension of $D_S$-modules associated to equation \ref{ext-PF} coincides with the one given by $\nu_{\xi}$ on $\Ext^{1}_{\text{MHM(S)}}(\mathbb{Q}_S(0), R^{2r-m-1}\pi_*(\mathcal{X}, \mathbb{Q}(r)))$, whereas the extension of local systems given by $\alpha_{\xi}$ coincides with the one given by $\nu_{\xi}$ on $\Ext^{1}_{\text{Perv(S)}}(\mathbb{Q}_S(0), R^{2r-m-1}\pi_*(\mathcal{X}, \mathbb{Q}(r)))$.
\end{proposition}

Moreover, Brosnan et. al (see \cite{brosnan}) have shown that the class $\alpha_{\xi}$ is actually a Tate class of weight 0 in $\mbox{IH}^1(S,\mathbb{H})$.


\begin{thebibliography}{10}

\bibitem{As}
M.~Asakura.
\newblock Motives and algebraic de {R}ham cohomology.
\newblock In {\em The arithmetic and geometry of algebraic cycles,
  {P}roceedings of the {CRM} summer school}, volume~24 of {\em CRM Proceedings
  and Lecture Notes}, pages 133--155, 2000.

\bibitem{brosnan}
P.~Brosnan, H.~Fang, Z.~Nie, and G.~Pearlstein.
\newblock Singularities of admissible normal functions.
\newblock {\em Invent. math.}, 177:599--629, 2009.

\bibitem{dAMS}
P.~L. del Angel and S.~M{\"{u}}ller-Stach.
\newblock The transcendental part of the regulator map for {K}1 on a mirror
  family of {K}3 surfaces.
\newblock {\em Duke Math. J.}, 112(3):581--598, 2002.

\bibitem{delAngelMS}
P.~L. del Angel and S.~M{\"{u}}ller-Stach.
\newblock Differential equations associated to families of algebraic cycles.
\newblock {\em Ann. {I}nst. {F}ourier ({G}renoble)}, 58(6):2075--2085, 2008.

\bibitem{GelfandManin}
S.~Gelfand and Y.~Manin.
\newblock {\em Methods of homological algebra}.
\newblock Springer, 2003.

\bibitem{Hernandez}
J.~Hern\'andez.
\newblock {\em Filtrations on higher chow groups and arithmetic normal
  functions}.
\newblock PhD thesis, University of Alberta, 2012.

\bibitem{KLMS}
M.~Kerr, J.~Lewis, and S.~M{\"{u}}ller-Stach.
\newblock The {A}bel-{J}acobi map for higher {C}how groups.
\newblock {\em Compositio Mathematica}, 142(2):374--396, 2006.

\bibitem{KL}
M.~Kerr and James~D. Lewis.
\newblock The {A}bel-{J}acobi map for higher {C}how groups, {II}.
\newblock {\em Inventiones Mathematicae}, 170(2):355--420, 2007.

\bibitem{LewisSaito}
J.~Lewis and S.~Saito.
\newblock Algebraic cycles and {M}umford-{G}riffiths invariants.
\newblock {\em American Journal of Mathematics}, 129(6):1449--1499, 2007.

\bibitem{PetersSteenbrink}
C.~Peters and J.~Steenbrink.
\newblock {\em Mixed {H}odge structures}.
\newblock Springer, 2008.

\bibitem{SaitoM5}
M.~Saito.
\newblock Modules de {H}odge polarisables.
\newblock {\em Publ. RIMS, Kyoto Univ.}, 24:849--995, 1988.

\bibitem{SaitoM3}
M.~Saito.
\newblock Mixed {H}odge modules.
\newblock {\em Publ. RIMS, Kyoto Univ.}, 26:221--333, 1990.

\bibitem{SaitoM4}
M.~Saito.
\newblock On the formalism of mixed sheaves.
\newblock Preprint, 2006.

\bibitem{stiller}
P.~Stiller.
\newblock {\em Automorphic Forms and the Picard Number of an Elliptic Surface}.
\newblock Vieweg, 1984.

\bibitem{Verdier}
J.-L. Verdier.
\newblock Des cat\'egories d\'eriv\'ees des cat\'egories ab\'eliennes.
\newblock {\em Ast\'erisque}, 239, 1996.

\bibitem{Weibel}
Charles~A. Weibel.
\newblock {\em An introduction to homological algebra}.
\newblock Cambridge University Press, 1994.

\end{thebibliography}
\end{document}